\documentclass[12pt,reqno]{amsart}

\title[Cohomology of fixed point varieties in Hilbert schemes of points on a surface]{Cohomology of fixed point sets of anti-symplectic involutions in the Hilbert scheme of points on a surface}
\author{Thomas John Baird}

\usepackage{amsmath,amsthm,amsfonts,amssymb,bbm,amscd,mathrsfs}
\usepackage[all]{xy}
\usepackage{sseq}
\usepackage{comment}
\usepackage{enumerate}
\usepackage{bm}
\usepackage{youngtab}
\usepackage[
            breaklinks=true,
            pdftitle={},
            pdfauthor={Thomas John Baird}]{hyperref}
\usepackage{color}
\definecolor{tocolor}{rgb}{.1,.1,.5}
\definecolor{urlcolor}{rgb}{.2,.2,.6}
\definecolor{linkcolor}{rgb}{.1,.4,.6}
\definecolor{citecolor}{rgb}{.6,.3,.1}
\hypersetup{colorlinks=true, urlcolor=urlcolor, linkcolor=linkcolor, citecolor=citecolor}

\newcommand{\Z}{\mathbb{Z}}

\newcommand{\C}{\mathbb{C}}
\newcommand{\Q}{\mathbb{Q}}
\newcommand{\PP}{\mathscr{P}}

\newcommand{\codim}{\textnormal{codim}}

\DeclareMathOperator{\supp}{\textnormal{Supp}}

\newtheorem{thm}{Theorem}[section]
\newtheorem{cor}[thm]{Corollary}
\newtheorem{lem}[thm]{Lemma}
\newtheorem{prop}[thm]{Proposition}

\theoremstyle{definition}

\numberwithin{equation}{section}

\begin{document}

\begin{abstract}

Let $S$ be a smooth, quasi-projective complex surface with complex symplectic form $\omega \in H^0(S, K_S)$.  This determines a symplectic form $\omega_n$ on the Hilbert scheme of points $S^{[n]}$ for $n \geq 1$.  Let $\tau$ be an  anti-symplectic involution of $(S,\omega)$:  an order two automorphism of $S$ such that $ \tau^*\omega=-\omega$. Then $\tau$ induces an anti-symplectic involution on $(S^{[n]},\omega_n)$ and the fixed point set $(S^{[n]})^\tau$ is a smooth Lagrangian subvariety of $S^{[n]}$.  In this paper, we calculate the mixed Hodge structure of $H^*( (S^{[n]})^\tau; \Q)$ in terms of the mixed Hodge structures of $H^*( S^\tau;\Q)$ and of $H^*( S / \tau; \Q)$.  We also classify the connected components of $(S^{[n]})^\tau$ and determine their mixed Hodge structures.  Our results apply more generally whenever $S$ is a smooth quasi-projective surface, and $\tau$ is an involution of $S$ for which $S^\tau$ is a curve.  
\end{abstract}

\maketitle


\section{Introduction}

Let $S$ be a smooth, complex quasi-projective surface.  A branching involution on $S$ is an automorphism $\tau \in Aut(S)$ of order two such that the fixed point set $C:= S^\tau$ is a curve. This implies that the quotient variety $ B := S/\tau$ is also smooth (Prop. \ref{branchinvo}) and the quotient morphism $q: S \rightarrow B$ is two-fold cover branched along $C$.  For example, if $(S,\omega)$ is symplectic surface, then an anti-symplectic involution on $(S,\omega)$ must be a branching involution.

Let $S^{[n]}$ be the Hilbert scheme of $n$ points in $S$. It is a smooth quasi-projective variety of dimension $2n$. A branching involution on $S$ induces an involution on $S^{[n]}$ with fixed point set $(S^{[n]})^\tau$ which is smooth of dimension $n$.  If $(S,\omega,\tau)$ is a symplectic surface with anti-symplectic involution, then $S^{[n]}$ inherits a symplectic form and an anti-symplectic involution, and the fixed point set $(S^{[n]})^\tau$ is a complex Lagrangian submanifold of in $S^{[n]}$.  If $C = \emptyset$, then $(S^{[2n]})^\tau \cong B^{[n]}$ and $(S^{[2n-1]})^\tau \cong \emptyset$ for $n \geq 1$. However, if $C \neq \emptyset$, then $(S^{[2n]})^\tau$ has more interesting geometry.

In this paper we calculate the cohomology groups $H^*( (S^{[n]})^\tau; \Q)$ with mixed Hodge structure, in terms of the cohomology of $B$ and $C$. We also classify the connected components of $(S^{[n]})^\tau$ and calculate the mixed Hodge structure of their cohomology. For a variety $X$, the mixed Hodge polynomial is defined $$MH(X) = MH(X; x,y,t) :=  \sum_{i,j,k} h_{i,j,k}(X) x^i y^j t^k$$ where $h_{i,j,k}(X)$ are mixed Hodge numbers (see (\ref{Hodgenumbers})). Our main theorem is a generating function for the mixed Hodge polynomial of $(S^{[n]})^\tau$.

\begin{thm}\label{mainthm}
If $(S,\tau)$ is a smooth quasi-projective surface equipped with a branching involution $\tau$ with $B = S/\tau$ and $C= S^\tau$, then
\begin{equation}\label{mainthmformula}
 \sum_{n=0}^\infty u^n MH( (S^{[n]})^\tau) = \Phi_2(B) \Phi_1(C)
 \end{equation}
where 
$$  \Phi_a(X) :=   \prod_{i,j,k, m \geq 0} \left(  1- (-1)^{k} x^{i+m} y^{j+m} t^{k+2m} u^{a+2m}  \right)^{(-1)^{k+1}h_{i,j,k}(X)}.  $$
\end{thm}

This can be compared with the mixed Hodge polynomial formula for $S^{[n]}$ calculated by G\"ottsche-Soergel \cite{GS}

\begin{equation}\label{soehur}
 \sum_{n=0}^\infty u^n MH(S^{[n]}) =  \prod_{i,j,k, m \geq 0} \left(  1- (-1)^{k} x^{i+m} y^{j+m} t^{k+2m} u^{1+m}  \right)^{(-1)^{k+1}h_{i,j,k}(S)}.
 \end{equation}
 
Formula (\ref{soehur}) was later explained by Nakajima \cite{N2} and Grojnowski \cite{G} by proving that $$ \mathbb{H}^{*,*}  := \bigoplus_{n\geq 0, k\geq 0} H^k( S^{[n]})$$ is the Fock space for a certain Heisenberg superalgebra associated to $H^*(S)$ with its intersection pairing. In particular, as a bigraded-vector space, $\mathbb{H}$ is isomorphic to the graded symmetric algebra $S^*(V)$ generated by the super vector space $V = V^{even} \oplus V^{odd}$ with $$V^{even/odd} = \Q \{p_1,p_2,...\}  \otimes H^{even/odd}(S)$$ where $p_n \otimes H^k(S)$ has bidegree $(n, 2n+k-2)$. Theorem \ref{mainthm} implies an isomorphism of bigraded vector spaces $$\mathbb{H}^{*,*}_\tau  := \bigoplus_{n\geq 0, k\geq 0} H^k( (S^{[n]})^\tau) \cong S^*(W),$$ with $$ W^{even/odd} :=\left( \Q\{q_2, q_4,...\} \otimes H^{even/odd}(B) \right) \oplus \left( \Q\{q_1, q_3,...\} \otimes H^{even/odd}(C) \right),$$
where $ q_n \otimes H^k(B) $ has bidegree $(n, n +k-2)$ and $ q_n \otimes H^k(C)$ has bidegree $(n, n+k-1)$. This suggests that $ \mathbb{H}_\tau$ may be the Fock space for a Heisenberg superalgebra determined by $H^*(B)$ and $H^*(C)$.  However, we do not investigate this question in the current paper.

To prove Theorem \ref{mainthm}, we modify the calculation by G\"ottsche-Soergel of $H^*(S^{[n]})$ in \cite{GS}. The Hilbert-Chow morphism $\pi: S^{[n]} \rightarrow S^{(n)}$ is semismall, which implies that the derived pushforward of the constant sheaf $\Q_{S^{[n]} }$ is perverse (up to a degree shift) and decomposes into a direct sum of intersection cohomology sheaves supported on certain subvarieties of the symmetric product $S^{(n)}$. The cohomology of these intersection cohomology sheaves can then be calculated in terms of the cohomology of symmetric products of $S$.  

Restricting to $\tau$-fixed points, we show that $ \pi_\tau: (S^{[n]})^{\tau} \rightarrow (S^{(n)})^\tau$ is also semismall. A similar argument then allows us to calculate $H^*( (S^{[n]})^\tau)$ in terms of the cohomology of symmetric products of $B$ and of $C$.

We briefly outline the paper.  

In \S \ref{Hilbert schemes and the Hilbert-Chow morphism},  we review Hilbert schemes and the Hilbert-Chow morphism.

In \S \ref{Branching involutions and fixed point Lagrangians}, we consider branching involutions and anti-symplectic involutions.

In \S \ref{The case S= C2}, we consider the case $S= \C^2$ with $\tau(x,y) = (-x,y)$. We calculate the Betti numbers of $(S^{[n]})^\tau$ using a Bialynicki-Birula stratification similar to the method used for $S^{[n]}$ by Ellingsrud and Str\o mme \cite{ES}.  In particular, we determine the dimension and cohomology groups of the Hilbert-Chow fibre $(F^n)^\tau:= \pi_\tau^{-1}( n (0,0))$. 

In \S \ref{Mixed Hodge structures}, we tackle the general case and prove Theorem \ref{mainthm}. This relies on applying the Decomposition Theorem for semismall morphisms to the restricted Hilbert-Chow morphism $\pi_\tau:  (S^{[n]})^\tau \rightarrow (S^{(n)})^\tau$. 

In \S \ref{Connected Components of}, we classify the connected components of $(S^{[n]})^\tau$ and calculate their mixed Hodge structures.

In \S \ref{Examples of anti-symplectic involutions on surfaces} we consider some examples and applications.

\textbf{Acknowledgements:} This research was supported by NSERC Discovery Grant RGPIN-2022-04908.

\section{Hilbert schemes and the Hilbert-Chow morphism}\label{Hilbert schemes and the Hilbert-Chow morphism}

For a complex quasi-projective variety $Y$ and $n\geq 1$, denote the symmetric product $ Y^{(n)} = Y^n/ \mathcal{S}_n$ where the symmetric group $\mathcal{S}_n$ acts on the product $Y^n$ by permuting factors. The variety $Y^{(n)}$ parametrizes effective 0-cycles of degree $n$, and its points can be written as formal sums $ \sum_{i} \nu_i y_i$  where $y_i \in Y$, $\nu_i \geq 0$ and $\sum \nu_i = n$.

If $X$ is a quasi-projective scheme over $\C$, then the Hilbert scheme of $n$ points on $X$, denoted $X^{[n]}$ or $Hilb^n(X)$, is the scheme that represents the functor
$$\underline{Hilb}_X^n:  Schemes \mapsto Sets$$
where 
$ \underline{Hilb}_X^n(U)$ is the set of closed subschemes $ Z \subset  X \times U$ for which the projection to $U$ is finite flat and surjective of degree $n$.  For a ring $R$ we write $Hilb^n(R)$ for $Hilb^n(Spec(R))$.

The closed points of $X^{[n]}$ correspond to closed subschemes $Z$ of $X$ of length $n$, meaning that the ring of functions $\mathcal{O}_Z(Z) $ is an algebra of dimension $n$ over $\C$. Any $Z \in X^{[n]}$ must be supported at finitely many points in $X$. The Hilbert-Chow morphism $\pi: X^{[n]}_{red} \rightarrow X^{(n)}$ is defined, at the level of sets, by $\pi(Z) :=  \sum_{x \in X}  \left( \dim_\C \mathcal{O}_{Z,x} \right) x$. 

It will be important to understand the fibres of $\pi :  S^{[n]}\rightarrow S^{(n)}$ when $S$ is a smooth quasi-projective complex surface (in this case $S^{[n]} = S^{[n]}_{red}$ is already reduced).  If $p \in S$ is a (closed) point, let $\mathcal{O}_p$ be its local ring with maximal ideal $\mathfrak{m}_p$.  Closed subschemes of $S$ of length $n$ supported at $p$ correspond to subschemes of length $n$ in $\mathcal{O}_p$, so there is a bijection $\pi^{-1}(n p) \cong Hilb^n(\mathcal{O}_p/ \mathfrak{m}_p^n)$. Since $S$ is a smooth surface, there exists a ring isomorphism $\mathcal{O}_p/ \mathfrak{m}_p^n  \cong \C[x,y]/(x,y)^n$. This determines a set-theoretic isomorphism 
$$F^n:=  Hilb^n(\C[x,y]/(x,y)^n)  \cong  \pi^{-1}(n p). $$
More generally, if $\sum_{i=1}^k \nu_i p_i  \in S^{(n)}$ with the $p_i$ distinct, then there is an isomorphism (\cite{F}  Prop. 2.3) $$ \pi^{-1}\left(\sum \nu_i p_i \right) \cong F^{\nu_1} \times...\times F^{\nu_k}.$$      

Suppose that $\tau$ is a branching involution of $S$. Denote $\pi_\tau$ the restriction of the Hilbert-Chow morphism to $\tau$-fixed points, $\pi_\tau:  (S^{[n]})^\tau \rightarrow (S^{(n)})^\tau$.  Elements of $(S^{(n)})^\tau$ have the form $$D:=  \sum_{i=1}^l  \lambda_i  (p_i + \tau(p_i))  + \sum_{j=1}^m \mu_j q_j,$$ where  $p_1,...,p_l, \tau(p_1),...,\tau(p_l), q_1,...,q_m \in S $ are all distinct points, $\tau(q_j) = q_j$ for all $j$, and the $\lambda_i, \mu_j$ are positive integers such that $ \sum_i 2\lambda_i + \sum_j \mu_j = n$. The fibre of $\pi_\tau$ above such a point is isomorphic to 
$$ \pi_\tau^{-1}(D) =  (\pi^{-1}(D) )^\tau  \cong  F^{\lambda_1} \times ... \times F^{\lambda_l} \times (F^{\mu_1})^\tau \times ...\times (F^{\mu_m})^\tau$$
where $(F^n)^\tau$ is the fixed point set of the action on $F^n = Hilb^n( \C[x,y]/\mu^n)$ under the automorphism $(x,y) \mapsto (-x,y)$.  We investigate $(F^n)^\tau$ in section \ref{The case S= C2}.

\section{Branching involutions and fixed point Lagrangians}\label{Branching involutions and fixed point Lagrangians}

Let $S$ be a smooth, connected, quasi-projective complex surface and let $S^{[n]}$ be its Hilbert scheme of points, which is smooth of dimension $2n$ \cite{F}. If $S$ admits a symplectic form $\omega \in H^0(X, \Omega^2)$, then there is an induced symplectic form $\omega_n$ on $S^{[n]}$ for all $n \geq 1$ \cite{B}.   An anti-symplectic involution on $(S,\omega)$ is an automorphism $\tau: S \rightarrow S$ such that $\tau \circ \tau = Id_S$ and $\tau^*(\omega) = -\omega$.

\begin{prop}\label{propforref}
If $(S,\omega,\tau)$ is as above, then the induced automorphism $\tau: S^{[n]}\rightarrow S^{[n]}$ satisfies $\tau^*( \omega_n) = - \omega_n$. Consequently, the fixed point variety $ (S^{[n]})^\tau$ is a smooth Lagrangian subvariety of $(S^{[n]}, \omega_n)$ of dimension $n$.
\end{prop}

\begin{proof}
Let $S^{(n)}_*$ be the subset of $S^{(n)}$ consisting of sums $\sum_{i=1}^n p_i$ with the $p_i$ distinct. Denote by $S^{[n]}_*$  the preimage of $S^{(n)}_*$ under the Hilbert-Chow morphism $\pi: S^{[n]}\rightarrow S^{(n)}$.  Then $\pi$ is an isomorphism from $S^{[n]}_*$  to $S^{(n)}_*$. 

Let $S^n = S \times ... \times S$ be the product space and let $\rho_i: S^n \rightarrow S$ be the projection morphisms for $i=1,...,n$. The symplectic form $\omega$ on $S$ determines a symplectic form  $\tilde{\omega}_n :=  \sum_{i=1}^n \rho_i^* \omega $ on $S^n$ which is invariant under the action of the symmetric group, so it descends to a symplectic form $\omega_{n,*}$ on $S^{(n)}_* \cong S^{[n]}_*$, and $\omega_n$ is the unique extension of this form to $S^{[n]}$.

The action of $\tau$ on $S$ induces an involution on  $S^{[n]}$ which restricts to an anti-symplectic involution on $(S^{[n]}_*, \omega_{n,*})$ hence it must be anti-symplectic on $(S^{[n]},\omega_n)$.
\end{proof}

More generally, for a smooth surface $S$, we call $\tau \in Aut(S)$ a branching involution if $\tau \circ \tau = Id_S$ and the fixed point locus $S^\tau$ is a (necessarily smooth) curve. Every anti-symplectic involution of a symplectic surface is a branching involution, because the fixed point locus $S^\tau$ is Lagrangian.   

\begin{prop}\label{branchinvo}
Let $S$ be a smooth quasi-projective surface and $\tau \in Aut(S)$ a branching involution. The quotient $B:= S/  \tau$ is a smooth surface, and the quotient map $q: S \rightarrow S/  \tau $ is a branched double cover with branch locus $C:= S^\tau$.
\end{prop}

\begin{proof}
Because $S$ is quasi-projective, the geometric quotient $S/ \tau$ exists as a variety (see \cite{S} III. 11). Clearly $S/ \tau $ is smooth at points where $\tau$ acts freely. So let $p \in S^\tau$ be a fixed point.  Since $S$ is smooth, the local ring $\mathcal{O}_p$ is regular of dimension two.  If $\mathfrak{m}_p$ is the maximal ideal, then $ \mathfrak{m}_p/ \mathfrak{m}^2_p$ is a 2-dimensional complex vector space which decomposes into $\pm 1$-eigenspaces of multiplicity one. We can therefore choose generators $f, h \in \mathfrak{m}_p$  so that $ f \circ \tau  = -f$ and $h \circ \tau = h$. The corresponding local ring for $S/ \tau $ is the invariant ring $\mathcal{O}_p^\tau$ which has maximal ideal $ \mathfrak{m}^\tau$ generated by two elements $f^2, h$ and consequently must be regular.
\end{proof}

\begin{prop}
Let $(S,\tau)$ be a branching involution.  Then the fixed point locus $ (S^{[n]})^\tau$ is a smooth variety of dimension $n$.
\end{prop}

\begin{proof}
Let $Z \in S^{[n]}$ satisfy $\tau( Z) = Z$.  It suffices to show that the isotropy action of $\tau$ on the tangent space $T_Z S^{[n]}$ invariant subspace of dimension $n$. 

Suppose that $\pi(Z) = \sum \nu_i p_i$  with $p_1,...,p_k$ distinct and let $Z = Z_1 \cup ... \cup Z_k$ be the corresponding decomposition. Then there is a natural isomorphism $ T_ZS^{[n]} \cong \bigoplus_{i=1}^k T_{Z_i} S^{[\nu_i]}$.   If $\tau(p_i) =p_j$ then $\nu_i = \nu_j$ and $\tau_*  ( T_{Z_i}S^{[\nu_i]}) = T_{Z_j}S^{[\nu_j]}$.  If $i \neq j$ then the invariant subspace of $  T_{Z_i} S^{[\nu_i]} \oplus  T_{Z_j} S^{[\nu_j]}$ is the graph of an isomorphism, so must be half dimensional.  

If $ i= j$, then $\tau$ induces an involution of $T_{Z_i} (S^{ [\nu_i]})$. We claim that $ T_{Z_i} (S^{ [\nu_i]})^\tau$ is half dimensional in $T_{Z_i} (S^{ [\nu_i]})$. Choose local coordinate functions $f,g$ on $S$ such that $f(p)=g(p)=0$ and $\tau(f)=-f$ and $\tau(g)=g$. This determines an isomorphism between $T_{Z_i} S^{\nu_i} \cong T_{Z_0}(\C^n)^{\nu_i}$ for some $Z_0 \in \pi^{-1}(\nu_i (0,0))$.  This isomorphism sends $\tau$ to the involution on induced by $\tau_0 (x,y) = (-x,y)$ on $\C^2$, which is anti-symplectic with respect to $dx \wedge dy$.  Applying Proposition \ref{propforref} proves the claim.
\end{proof}

\subsection{Special components of $(S^{[n]})^\tau$}

Since $C = S^\tau \subset S$ is a closed subscheme, we obtain a natural transformation of functors $\underline{Hilb}_C^n \Rightarrow \underline{Hilb}_S^n$ and consequently a morphism $C^{[n]} \rightarrow S^{[n]}$ which maps into $(S^{[n]})^\tau$. 

We also may define a natural transformation $  \underline{Hilb}_B^n \Rightarrow \underline{Hilb}_S^{2n}$   by the rule that $Z \subset  B\times U$ is sent to $\tilde{Z} \subseteq S \times U$ determined by the fibre product $$  \xymatrix{   \tilde{Z} \ar[r] \ar[d]  &  S \times U \ar[d]  \\   Z \ar[r] & B \times U}. $$
Both flatness and finiteness is preserved under base change, and the degree of $\tilde{Z} \rightarrow U$ is 2n because it is the composition of the degree two morphism $\tilde{Z} \rightarrow Z$ with the degree $n$ morphism $Z \rightarrow U$. We obtain a morphism $B^{[n]} \rightarrow S^{[2n]}$. 

Both morphisms $C^{[n]} \rightarrow (S^{[n]})^\tau$ and $B^{[n]} \rightarrow (S^{[2n]})^\tau$ are injective. Clearly we have an isomorphism $ C^{[1]} \cong  (S^{[1]})^\tau$ and it follows from Theorem \ref{mainthm} that we obtain an isomorphism $ C^{[2]} \coprod B^{[1]} \cong  (S^{[2]})^\tau$.  If $C = \emptyset$, then $B^{[n]} \cong S^{[2n]}$ and $S^{[2n-1]} =  \emptyset$ for all $n\geq 1$.  However, if $ C \neq \emptyset$ and $n \geq 3$ then there are other components of $(S^{[n]})^\tau$ not obtained by these constructions.

\section{The case $S= \C^2$}\label{The case S= C2}

In this section, let $S = \C^2 = \mathrm{Spec}( \C[x,y])$. We begin by reviewing the calculation of the Betti numbers of $H^*(S^{[n]})$ following Ellingsrud-Str\o mme \cite{ES} (see also \cite{N}, Chapter 5).

The torus $T^2 = \C^\times \times \C^\times$ acts on $S$ by $(t_1, t_2) \cdot (x,y) = (t_1 x, t_2 y)$. This induces a $T^2$-action on $S^{[n]}$ and on $F^n := \pi^{-1}(n (0,0))$.  The induced $T^2$-action on $\C[x,y]$ has monomials $x^my^n$ as eigenvectors. Consequently, the fixed point set $(S^{[n]})^{T^2} = (F^n)^{T^2}$ is in one-to-one correspondence with the set of monomial ideals in $\C[x,y]$ of colength $n$, which in turn is in one-to-one correspondence with $ \PP_n$ the set of partitions of $n$. To a partitition $\lambda = ( \lambda_1 \geq \lambda_2 \geq ... \geq \lambda_k) \in \PP_n $ with $\lambda_1 +...+\lambda_k = n$, we associate the ideal $$I_\lambda := (y^{\lambda_1}, x y^{\lambda_2},..., x^{\ell(\lambda)}),$$ where $\ell(\lambda) = k$ is the length of $\lambda$. The subschemes $Z_\lambda := Spec( \C[x,y]/I_\lambda)$ are the $T^2$-fixed points in $(S^{[n]})^{T^2}$.

Fix an integer $N \gg n$ and the corresponding one parameter subgroup  $G:= \{(t^N, t) \in T^2 | t \in \C^\times\} \leq T^2$. The restricted action by $G$ on $S^{[n]}$ fixes the same finite set of points as $T^2$. Define for each $\lambda \in \PP_n$ the stable/unstable varieties 
$$  W_\lambda^+  := \{ p \in S^{[n]} | \lim_{t\rightarrow 0} (t^N, t) \cdot p  = Z_\lambda \}  $$ 
$$  W_\lambda^-  := \{ p \in S^{[n]} | \lim_{t\rightarrow \infty} (t^N, t) \cdot p = Z_\lambda \} .$$ 
Each stratum $W^{\pm}_{\lambda}$ is isomorphic to an affine space $ \mathbb{C}^{d^{\pm}_\lambda}$ and $d^{\pm}_\lambda$ is the dimension of the positive/negative weight spaces for the isotropy action of $G$ on $T_{Z_\lambda} S^{[n]}$ (implying that $d_\lambda^+ + d_\lambda^- = 2n$). These determine Bialynicki-Birula stratifications
\begin{eqnarray*}
 S^{[n]} = \bigcup_{\lambda \in \PP_n} W_\lambda^+,  && F^n  = \bigcup_{\lambda \in \PP_n} W_\lambda^-. 
 \end{eqnarray*}
 
Topologically, this means that $F^n$ is a cell complex, with a cell of dimension $2 d^-_\lambda$ for each $\lambda \in \PP_n$. Since all cells have even real dimension, they determine a basis for the homology of $F^n$. Similarly, the one-point compactification of $S^{[n]}$ is a cell complex with a cell of dimension $2d^+_\lambda$ for each $\lambda \in \PP_n$ (plus a point at infinity).  But since $d^+_\lambda + d^-_\Lambda = 2n$, Poincar\'e duality implies the Betti numbers of $S^{[n]}$ and $F^n$ agree.  Indeed $F^n$ is a deformation retract of $S^{[n]}$. 

The representation ring for $T^2$ is isomorphic to the polynomial ring  $\Z [ T_1, T_2]$  where $T_i$ is the one dimensional representation $ (t_1,t_2) \mapsto t_i$.   
The isotropy action of $T^2$ on the tangent space at $Z_{\lambda}$ decomposes as

$$T_{Z_\lambda} S^{[n]} = \sum_{s \in D_\lambda} T^{l(s)+1}_1 T_2^{-a(s)} + T_1^{-l(s)}T_2^{a(s)+1} $$
where the sum is indexed by boxes $s$ in the Young diagram $D_\lambda$ of $\lambda$ and $a(s), l(s)$ are the arm length and leg length respectively. 
 
Restricting to the isotropy action of $G$ on $T_{Z_\lambda} S^{[n]}$, the negative weight spaces at $Z_\lambda$ are precisely those terms of the form $T_1^{-l(s)}T_2^{a(s)+1} $ when $l(s)>0$. Therefore $$d_\lambda^- =  (\lambda_1 -1) + (\lambda_2-1)+... =  n -\ell(\lambda).$$ 

Denoting the Poincar\'e polynomial of a space $X$ by $ P_t ( X) =  \sum_{i \geq 0}  \dim H^i(X,\Q) t^i$, we have  
\begin{equation}\label{snpp}
 P_t( F^n) =  P_t( S^{[n]}) = \sum_{\lambda \in \PP_n} t^{ 2 d_\lambda^-} =  \sum_{\lambda \in \PP_n} t^{ 2 (n - \ell(\lambda))}.
 \end{equation} 
In particular, since $F^n$ is projective we deduce that
\begin{equation}\label{dimfn}
\dim(F^n) = n-1
\end{equation}
and $F^n$ has only one irreducible component of that dimension  (in fact it is irreducible). 

We derive from (\ref{snpp}) the generating function
$$  \sum_{n\geq 0}  u^n P_t( F^{n}) =  \sum_{n\geq 0}  u^n P_t( S^{[n]}) =  \sum_{\lambda \in \PP} t^{2(|\lambda| - \ell(\lambda))}u^{|\lambda|} =  \prod_{k \geq 1}  \frac{1}{1-t^{2k-2}u^k}.$$

Consider now the action of  $\tau$ on $S$ by $\tau(x,y) = (-x, y)$, which is anti-symplectic with respect to the symplectic form $\omega = dx \wedge dy$. Since $\tau$ is induced by $(-1,1) \in T^2$, the $T^2$-action on $S^{[n]}$ restricts to one on $(S^{[n]})^\tau$ and we have equality of fixed point sets 
$$ (( S^{[n]})^\tau)^{T^2} = (S^{[n]})^{T^2}  = \{ Z_\lambda | \lambda \in \PP_n\},$$ 
so by the same argument as before, we acquire Bialynicki-Birula stratifications of both $(F^n)^\tau$ and $(S^{[n]})^\tau$ indexed by $\PP_n$,
\begin{eqnarray}\label{stratoftauinv}
 (S^{[n]})^\tau = \bigcup_{\lambda \in \PP_n} W_\lambda^{+,\tau},  && F^n  = \bigcup_{\lambda \in \PP_n} W_\lambda^{-,\tau}. 
 \end{eqnarray}
Each stratum $W^{\pm,\tau}_{\lambda}$ is isomorphic to an affine space $\mathbb{C}^{d^{\pm,\tau}_\lambda}$ where $d^{\pm,\tau}_\lambda$ is the dimension of the positive/negative weight spaces for the isotropy action of $G$ on $T_{Z_\lambda} (S^{[n]})^\tau$ (which implies that $d_\lambda^{+,\tau} + d_\lambda^{-,\tau} = n$).

The isotropy action of $T^2$ on the tangent space of a fixed point $Z_\lambda \in  (S^{[n]})^\tau$ is 
$$T_{Z_\lambda} (S^{[n]})^\tau= \sum_{\substack{ s \in D \\ l(s) \mathrm{~is~odd}}} T^{l(s)+1}_1 T_2^{-a(s)} +  \sum_{\substack{ s \in D \\ l(s) \mathrm{~is~even}}} T_1^{-l(s)}T_2^{a(s)+1} .$$
The negative definite summands for the $G$-action are those terms $T_1^{-l(s)}T_2^{a(s)+1}$ when $l(s)$ is even and $l(s)>0$. Therefore $d_\lambda^{-,\tau} = \lfloor (\lambda_1-1)/2 \rfloor +   \lfloor (\lambda_2-1)/2 \rfloor+...$ where $\lfloor . \rfloor$ means greatest integer part. Consequently,
\begin{equation}\label{Ppsom}
 P_t( (F^n)^\tau) =  P_t( (S^{[n]})^\tau )  = \sum_{\lambda \in \PP_n} t^{2 d_\lambda^{-,\tau}} = \sum_{\lambda \in \PP_n} t^{ n - \ell(\lambda) -\ell_{even}(\lambda)}. 
 \end{equation} 
 where $\ell_{even}(\lambda)$ is the number of even parts of $\lambda$. In particular, $P_t( (F^n)^\tau)$ has degree equal to $n-1$ if $n$ is odd and $n-2$ if $n$ is even. Since $(F^n)^\tau$ is a projective variety, this determines its dimension
\begin{eqnarray}\label{maxcompfixedpoint}
 \dim  (F^n)^\tau & \leq&   \begin{cases} (n- 1)/2 & \text{ if $n$ is odd} \\  (n- 2)/2 &  \text{ if $n$ is even} \end{cases}
\end{eqnarray} 
and 

\begin{lem}\label{maxcompfixedpoint2}
When $n$ is odd,  the only irreducible component of top dimension in $(F^n)^\tau$, is the closure of the top dimensional cell $\overline{W_{(n)}^{-,\tau}}$. 
\end{lem}

From (\ref{Ppsom}) we also derive the generating function
\begin{eqnarray*}
  \sum_{n\geq 0}  u^n P_t( (F^n)^\tau )  = \sum_{n\geq 0}  u^n P_t( (S^{[n]})^\tau ) &=& \prod_{k \mathrm{~odd}}  \frac{1}{1-t^{k-1}u^k}  \prod_{k \mathrm{~even}}  \frac{1}{1-t^{k-2}u^k}  \\
 &=& \prod_{m\geq 0} \frac{1}{ (1-t^{2m}u^{2m+1}) (1-t^{2m}u^{2m+2})} ,
 \end{eqnarray*}

\section{Mixed Hodge structures}\label{Mixed Hodge structures}
In this section we apply the strategy of G\"ottsche-Soergel to calculate the mixed Hodge polynomial of $(S^{[n]})^\tau$.

Given a variety $X$, denote by $MHM(X)$ its category of mixed Hodge modules, by $D(X) := D^b(MHM(X))$ the bounded derived category, and by $\mathcal{H}^i: D(X) \rightarrow MHM(X)$ the cohomology functors. Objects $M$ in $D(X)$ admit degree shifts $M \mapsto M[d]$ and Tate twists $M \mapsto M(d)$ for $d \in \Z$, with $MHM(X)$ being stable under Tate twists. A morphism $f: Z\rightarrow Y$ determines a pair of adjoint functors $(f_*,f^*)$, relating $D(Z)$ and $D(Y)$. If $pt$ is point, then $MHM(pt)$ is the category of graded, polarizable mixed Hodge structures over $\Q$. Denote by $\Q_{pt}$ the one dimensional mixed Hodge structure of type $(0,0)$ and, if $p: X \rightarrow pt$, put $ \Q_X := p^* \Q_{pt}$.  Then $\mathcal{H}^i p_*(\Q_X) = H^i(X, \Q)$ equipped with its standard mixed Hodge structure.

For $X$ irreducible, denote by $IC(X) \in MHM(X)$ its intersection cohomology sheaf. It is characterized by the property that $IC(X)$ is simple in $MHM(X)$ and that if $i: U \hookrightarrow X$ is a dense, open, smooth subvariety, then $i^*IC(X) = \Q_U [dim U]$. For each irreducible subvariety $i: Y \hookrightarrow X$ we abuse notation and denote by $IC(Y) = i_*IC(Y) \in MHM(X)$.

Let $\phi: Z \rightarrow Y$ be a surjective proper morphism between quasi-projective varieties. Suppose that $ Z = \coprod Z_i$ is a disconnected union of finitely many irreducible varieties $Z_i$ of equal dimension. Suppose $Y$ admits a stratification into a finite number of irreducible nonsingular subvarieties $Y = \cup_y O_y$, where $y$ denotes a distinguished point in $O_y$. For each stratum $O_y$, assume that the restriction of $\phi$ to $\pi^{-1}(O_y)$ is a topological fibre bundle with base $O_y$ and fibre $\phi^{-1}(y)$. The morphism $\phi$ is called \emph{semismall} if it satisfies $ 2 \dim \phi^{-1}(y) \leq \codim O_y$ for every stratum $O_y$. We say $O_y$ is \emph{relevant} for $\phi$ if equality holds. 

The following is due to Borho and MacPherson \cite{BM} for perverse sheaves (see also \cite{N} Thm 6.8), and a version for mixed Hodge modules is proven by G\"ottsche-Soergel (\cite{GS} Thm 5). \footnote{G\"ottsche-Soergel prove their result in the special case of the the Hilbert-Chow morphism $\pi: S^{[n]} \rightarrow S^{(n)}$, but their argument applies almost unchanged in this more general setting.}

\begin{thm}\label{BorhoMac}
Let $\phi: Z \rightarrow Y$ be a semismall projective morphism with respect to the stratification $Y = \cup_y O_y$, and let $d_y := \dim \phi^{-1}(y)$.  Suppose that $\dim H^{2 d_y}( \phi^{-1}(y)) = 1$ for all relevant strata $O_y$. Then there is a canonical isomorphism in $MHM(Y)$ 
$$ \phi_* IC(Z) \cong  \bigoplus_{y}  IC( \overline{O}_y) ( -d_y )$$
where the summation runs through all relevant strata $O_y$ and $\overline{O}_y$ is the closure of $O_y$.
\end{thm}

We will apply this theorem to the restricted Hilbert-Chow morphism $ \pi_\tau:   (S^{[n]})^\tau \rightarrow (S^{(n)})^\tau$. Let $C = \cup_{k=1}^c C_k$ be the decomposition into connected components of $C = S^\tau$.  Choose $\lambda \in \PP$ and $\mu^k \in \PP$ for $k=1,...,c$  such that $2 | \lambda | + \sum_k |\mu^k| = n$. Denote the ordered tuple $\vec{\mu} = (\mu^1,...,\mu^c)$,  and write $| \vec{\mu}| =  |\mu^1|+...+|\mu^c|$, and $\ell(\vec{\mu}) = \ell(\mu^1)+...+\ell(\mu^k)$.  Define the locally closed subvariety $(S^{(n)})^\tau_{\lambda, \vec{\mu}}   \subseteq  (S^{(n)})^\tau$ consisting of divisors of the form 
 $$ D = \sum_{i}  \lambda_i  ( p_i   + \tau (p_i))  +  \sum_{k} \sum_{j}  \mu_j^k  q_j^k$$
 where $p_1, \tau(p_1),..., $  are distinct points in $S \setminus C$  and $q_1^k,q_2^k,...$ are distinct points in $C_k$.

\begin{prop}
The restricted Hilbert-Chow morphism $ \pi_\tau:   (S^{[n]})^\tau \rightarrow (S^{(n)})^\tau$ is semismall with respect to the stratification $$(S^{(n)})^\tau = \bigcup_{\substack{ \lambda \in \PP, \vec{\mu} \in \PP^c \\ 2|\lambda|+|\vec{\mu}| =n}} (S^{(n)})^\tau_{\lambda, \vec{\mu}} .$$ The relevant strata are those $(S^{(n)})^\tau_{\lambda, \vec{\mu}}$ for which the $\mu^k$ have only odd parts and the fibres above the relevant strata have one-dimensional cohomology in top degree. We deduce a canonical isomorphism of mixed Hodge modules
$$ \pi_* IC((S^{[n]})^\tau) =  \bigoplus_{\substack{ \lambda \in \PP, \vec{\mu} \in \PP^c_{odd} \\ 2|\lambda|+|\vec{\mu}| =n}} IC( \overline{(S^{(n)})^\tau}_{\lambda, \vec{\mu}}) \left(  \frac{2 \ell(\lambda)+\ell(\vec{\mu})-n}{2} \right)$$
summed over all partitions $\lambda$ and $c$-tuples of partitions $\vec{\mu}$ such that $2|\lambda|+|\vec{\mu}| = n$ and $\vec{\mu}$ has only odd parts.
\end{prop}

\begin{proof}
The strata $(S^{(n)})^\tau_{\lambda, \vec{\mu}} $ are clearly non-singular, connected and have codimension $ n - 2 \ell(\lambda) - \ell(\vec{\mu})$. If $D \in (S^{(n)})^\tau_{\lambda, \vec{\mu}}$, then the preimage is homeomorphic to 

\begin{eqnarray*}
\pi^{-1}_\tau(D) & \cong &  \prod_{i} \pi_{\tau}^{-1} (\lambda_i (p_i + \tau p_i)) \times \prod_{j,k} \pi_{\tau}^{-1} (\mu_j^k q_j^k) \\
& \cong &  \prod_{i} F^{\lambda_i} \times \prod_{j,k} \left(F^{\mu_j^k}\right)^{\tau}.
\end{eqnarray*}

We know (\ref{dimfn}) and (\ref{maxcompfixedpoint}) $$  \dim(F^{\lambda}) = \lambda - 1, \qquad   \dim((F^{\mu})^\tau) = \begin{cases} (\mu - 1)/2 & \text{ if $\mu$ is odd} \\  (\mu- 2)/2 &  \text{ if $\mu$ is even,}  \end{cases} $$
so
 \begin{eqnarray*}
  \dim \pi_\tau^{-1}(D)  & \leq &  \left( \sum_{i} (\lambda_i -1)  + \frac{1}{2} \sum_{j,k} \mu_j^k -1 \right) \\
  &= & n/2 - \ell(\lambda) - \ell(\vec{\mu})/2
 \end{eqnarray*}
with equality if and only if all of the $\mu_j^k$ are odd. Finally, if $D$ lies in a relevant stratum,

\begin{eqnarray*}
H^{n - 2 \ell(\lambda) + \ell(\vec{\mu}) }( \pi^{-1}(D)) & \cong&  \bigotimes_{i} H^{2(\lambda_i-1)}(F^{\lambda_i}) \otimes  \bigotimes_{j,k} H^{(\mu_j^k-1)}(\left(F^{\mu_j^k}\right)^{\tau}) \\
 & \cong&  \Q.
\end{eqnarray*}
\end{proof}

Next, we compute the intersection cohomology for the closures of relevant strata.   We borrow the following from (\cite{GS}  Lemma 1 and Proposition 3).

\begin{lem}\label{passingtoprods}
 Let $\kappa: X \rightarrow Y$ be a finite birational morphism between irreducible varieties. Then $\kappa_*IC(X) = IC(Y)$ in $D(Y)$.
\end{lem}

\begin{lem}\label{rathomman}  Let the finite group $G$ act on the smooth irreducible variety $X$. Then
$$IC(X/G) = \Q_{X/G}[\dim X]$$
in $D(X/G)$.
\end{lem}

For a variety $M$ and partition $\lambda = (1^{\alpha_1} 2^{\alpha_2}...)$ define $$M^{(\lambda)} := M^{(\alpha_1)} \times M^{(\alpha_2)} \times ....$$ 
Note that $M^{(\lambda)}$ is the quotient of $M^{\ell(\lambda)}$ by the finite permutation group $\mathcal{S}_{\alpha_1} \times \mathcal{S}_{\alpha_2} \times ...$. For a tuple of partitions $\vec{\mu}= (\mu^1,...,\mu^c)$ we write $ C^{(\vec{\mu})} = C_1^{(\mu^1)}\times C_2^{(\mu^2)}\times ...\times C_c^{(\mu^c)}$. 
 
With $ \lambda, \vec{\mu}$ as above,  consider the morphism $\psi:  B^{(\lambda)}\times C^{(\vec{\mu})}   \rightarrow S^{(n)}$  defined by $ \psi = \psi_B + \sum_{k=1}^c\psi_k$  where 
$$\psi_B: B^{(\lambda)} \rightarrow S^{(n)},  (D_1,D_2,...) \mapsto q^{-1}(D_1) + 2 q^{-1}(D_2) +... $$ 
and 
$$\psi_k: C_k^{(\mu^k)} \rightarrow S^{(n)},  (D_1,D_2,...) \mapsto D_1 + 2 D_2 + ... .$$
Then $\psi$ is a finite and birational morphism from  $B^{(\lambda)}\times C^{(\vec{\mu})}$ onto $ \overline{(S^{(n)})^\tau}_{\lambda, \vec{\mu}}$   so
 $$\psi_* IC\left( B^{(\lambda)} \times  C^{(\vec{\mu})}  \right) \cong IC(\overline{(S^{(n)})^\tau}_{\lambda, \vec{\mu}})$$ by Lemma \ref{passingtoprods}.  Since $B^{(\lambda)} \times C^{(\vec{\mu})}$ is the quotient of the non-singular variety $ B^{\ell(\lambda)} \times C^{\ell(\vec{\mu})}$ by a finite group, so 
$$  IC\left( B^{(\lambda)} \times C^{(\vec{\mu})} \right) \cong \Q_{ B^{(\lambda)} \times C^{(\vec{\mu})} } \left[ 2|\lambda| + |\vec{\mu}| \right].$$
by Lemma \ref{rathomman}. 

We deduce the canonical isomorphism of mixed Hodge structures

\begin{equation}\label{canondecompofMHS}
H^{*} ( (S^{[n]})^\tau ) \cong  \bigoplus_{\substack{ \lambda \in \PP, \vec{\mu} \in \PP_{odd}^c \\ 2|\lambda|+|\vec{\mu}| =n}}    H^{ * + 2\ell(\lambda) +\ell(\vec{\mu})-n}(  B^{(\lambda)} \times C^{(\vec{\mu})}) \left(  \frac{2 \ell(\lambda)+\ell(\vec{\mu})-n}{2}\right).
\end{equation}

We may coarsen the above decomposition by observing that for any $m \geq 0$ that 
$$ \coprod_{\substack{ \vec{\mu} \in \PP_{odd}^c \\ |\vec{\mu}|=m }} C^{(\vec{\mu})} = \coprod_{\substack{ \mu \in \PP_{odd} \\ |\mu|=m}} C^{(\mu)}, $$
so
$$ H^{*} ( (S^{[n]})^\tau ) \cong  \bigoplus_{\substack{ \lambda \in \PP, \mu \in \PP_{odd} \\ 2|\lambda|+|\mu| =n}}    H^{ * + 2\ell(\lambda) +\ell(\mu)-n}(  B^{(\lambda)} \times C^{(\mu)}) \left(  \frac{2 \ell(\lambda)+\ell(\mu)-n}{2}\right). $$

We use this formula to produce mixed Hodge polynomials. If $X$ is a quasi-projective variety, then for each $k \geq 0$ the singular cohomology $H^k(X;\C)$ admits a mixed Hodge structure, determining the Hodge filtration $$H^k(X; \C) = F_0 \supseteq F_1 \supseteq ...$$ and weight filtration $$ H^k(X; \C) \supseteq... \supseteq W_1 \supseteq W_0 \supseteq  W_{-1} ....$$ Define the mixed Hodge polynomial $$MH(X) = MH(X; x,y,t) := \sum_{i,j,k} h_{i,j,k}(X) x^i y^j t^k$$ with coefficients the mixed Hodge numbers 
\begin{equation}\label{Hodgenumbers}
h_{i,j,k}(X) :=  \dim_\C Gr^i_F Gr_{i+j}^W H^k(X;\C),
\end{equation} 
the dimensions of associated graded components of the bi-filtration.

From (\ref{canondecompofMHS}) we derive the generating function formula
\begin{eqnarray*}
 \sum_{n=0}^\infty u^n MH((S^{[n]})^\tau) &=&  \sum_{\lambda \PP,\mu \in \PP_{odd}}   u^{2|\lambda|+|\mu|} (xyt^2)^{ (n - 2 \ell(\lambda) -\ell(\mu) )/2}MH(B^{(\lambda)}) MH( C^{(\mu)})\\
 &= & \left( \sum_{\lambda \in \PP} u^{ 2 |\lambda|} (xyt^2)^{\left(|\lambda|-\ell(\lambda)\right)} MH(B^{(\lambda)}) \right)\left( \sum_{
 \mu \in \PP_{odd}} u^{  |\mu|}(xyt^2)^{(|\mu|-\ell(\mu))/2} MH(C^{(\mu)}) \right)
 \end{eqnarray*}

By MacDonald's formula \cite{M}, we know

$$ \sum_{n=0}^\infty u^n MH( X^{(n)}) =  \prod_{i,j,k \geq 0} \left(   1+(-1)^{k+1} x^i y^j t^{k}u \right)^{(-1)^{k+1}h_{i,j,k}(X)}. $$

It follows that

\begin{multline*}
 \sum_{\lambda \in \PP} u^{ 2 |\lambda|} (xyt^2)^{\left(|\lambda|-\ell(\lambda)\right)} MH(B^{(\lambda)})  \\
    =  \sum_{\lambda = (1^{a_1} 2^{a_2}...)} u^{2(a_1+2a_2+3a_3...)} (xyt^2)^{ a_2 + 2a_3+3a_4+...} MH(B^{(a_1)}) MH(B^{(a_2)})...  \\ 
    =  \prod_{m\geq 0} \sum_{a\geq 0} (u^{2m+2} (xyt^2)^{m})^a MH(B^{(a)})\\
    = \prod_{i,j,k, m \geq 0} \left(  1+ (-1)^{k+1} x^{i+m} y^{j+m} t^{k+2m} u^{2+2m}  \right)^{(-1)^{k+1}h_{i,j,k}(B)}.
\end{multline*}

Similarly

\begin{multline*}
 \sum_{\mu \in \PP_{odd}} u^{|\mu|} (xyt^2)^{(|\mu|-\ell(\mu))/2} MH(C^{(\mu)})  \\ 
 =  \sum_{\mu = (1^{a_1} 3^{a_3} ...)} u^{(a_1+3a_3+ 5a_5...)}(xyt^2)^{ a_3+2a_5+3a_7...} MH(C^{(a_1)}) MH(C^{(a_3)})...  \\ 
 =  \prod_{m\geq 0} \sum_{a\geq 0} (u^{2m+1} (xyt^2)^{m})^a MH(C^{(a)})\\
=  \prod_{i,j,k, m \geq 0} \left(   1+ (-1)^{k+1} x^{i+m} y^{j+m} t^{k+2m} u^{1+2m}  \right)^{(-1)^{k+1}h_{i,j,k}(C)}.
\end{multline*}

These combine to prove Theorem \ref{mainthm}.

\section{Connected Components of $(S^{[n]})^\tau$}\label{Connected Components of}
Let $S$ be a smooth quasi-projective surface and $\tau$ a branching involution with $C = S^\tau$ non-empty.  In this section, we classify the connected components of $(S^{[n]})^\tau$ and derive formulas for their mixed Hodge structures.  

\begin{lem}\label{cardcalc}
If $ S^\tau$ has $c>0$ connected components, then the number of connected components of $(S^{[n]})^\tau$ is $ \binom{n+c-1}{c-1} + \binom{n+c-3}{c-1} +... = \sum_{i=0}^{ [ n/2]} \binom{n+c-1-2i}{c-1}$ . 
  \end{lem}

\begin{proof}
Simply set $x=y=1$ and $t=0$ into (\ref{mainthmformula}) to get
$$ \sum_{n=0}^{\infty}  \dim_{\Q} H^0( S^{[n]})  u^n =  \frac{1}{(1-u^2)(1-u)^c} = \frac{1}{1-u^2} \left( \sum_{n=0}^{\infty}  \binom{n+c-1}{c-1} u^n \right)  $$
and the formula follows.
\end{proof}

For subsets $ U \subset Z$ define $$\mathcal{O}_Z( U) := \oplus_{p \in U}  \mathcal{O}_{Z,p}.$$
If $Z \in (S^{[n]})^\tau$ and $\tau(U) = U$, then  $\mathcal{O}_Z( U) $ decomposes into $\pm 1$-eigenspaces and we define the index $$Ind_Z(U):= \dim_\C \mathcal{O}_Z(U)_+ - \dim_\C \mathcal{O}_Z(U)_+.$$  Since $\tau$ sends $\mathcal{O}_{Z,p}$ to $\mathcal{O}_{Z,\tau(p)}$, points outside of $C = S^\tau$ contribute nothing to the index, so $$Ind_Z(U)=  \sum_{p \in U \cap C}  Ind_Z(p) = Ind_Z(U \cap C). $$

\begin{lem}
The function $Ind:  (S^{[n]})^\tau \rightarrow \Z^c$,  $Ind(Z) = ( Ind_Z(C_1),...,Ind_Z(C_c))$ is continuous with respect to the analytic topology on $(S^{[n]})^\tau$ and the discrete topology on $\Z^c$.
\end{lem}

\begin{proof}
It is enough to prove that $Ind_Z(C_k): (S^{[n]})^\tau \rightarrow \Z $ is continuous for all $k$.  

Let $\widetilde{S^{[n]}} \subseteq S^{[n]} \times S$ be the universal subscheme. The (ordinary) pushforward of $\mathcal{O}_{\widetilde{S^{[n]}}}$ under the projection to $S^{[n]}$ is a $\tau$-equivariant vector bundle $E$ with fibre $E_Z = \mathcal{O}_Z(Z)$. 

Any $Z_0 \in (S^{[n]})^\tau$ has finite support $\supp(Z_0)$, we can choose an analytic tubular open neighbourhood $C \subset U \subset S$, for which  $U \cap C = C_k$ and $\supp(Z_0)  \cap \partial U  = \emptyset$, where $\partial U: = \overline{U} \setminus U$. Then $$V := \{ Z  \in (S^{[n]})^\tau | \supp(Z) \cap \partial U = \emptyset\}$$ is an open neighbourhood of $Z_0$.  We have a vector subbundle $F \subseteq E|_{V}$ of locally constant rank, with fibres $F_Z := \mathcal{O}_Z(U)$, and $F$ is equipped with a continuous fibre-wise action by $\tau$.  The $\tau$-invariant subbundle $F^\tau$ therefore also has locally constant rank on $V$. Therefore $Ind_Z(C_k) = Ind_Z(U ) =  2 \dim_\C  F_Z^\tau - \dim_\C F_Z$ is continuous on $V$, hence also on $(S^{[n]})^\tau$.
\end{proof}

Define
$$I_{n,c} :=  \{ (n_1,...,n_c) \in \Z^c_{\geq 0}  |  \left( n-\sum_{k=1}^c n_i \right) \text{ is even and non-negative}\} .$$

\begin{prop}
If $S^\tau$ is non-empty, then the range of $Ind:  (S^{[n]})^\tau \rightarrow \Z^c$ is equal to $I_{n,c}$ and determines a bijection between $\pi_0( (S^{[n]})^\tau)$ and $I_{n,c}$.
\end{prop}

\begin{proof}

There is an obvious bijection between $I_{n,c}$ and $ \cup_{i\geq 0}  P_{n-2i,c} $ where $P_{n-2i,c}$ is the set of ordered $c$-partitions of $n-2i$. Therefore, by Lemma \ref{cardcalc},  $\pi_0((S^{[n]})^\tau)$ and $I_{n,c}$ have equal cardinality and it suffices to show that $I_{n,c}$ is contained in the range of $Ind$. 

Choose $(n_1,...,n_c)\in I_{n,c}$ and set $n_0 = (n - \sum_{k=1}^c n_k)/2$. Choose points $p_1,...,p_{n_0}, \tau(p_1),...,\tau(p_{n_0}) \in S \setminus C$,  and $q_1^k,...,q_{n_k}^k \in C_k$ all distinct.  Then the cycle $$D = \sum_{i} p_i +\tau(\pi_i) + \sum_{j,k} q_j^k$$ determines a unique scheme $Z_D \in (S^{[n]})^\tau$ such that $\pi( Z_D)= D$.  Then $\mathcal{O}_{Z_D}(C_k) \cong \C^{n_k}$ with the trivial $\tau$-action, so  $Ind(Z_D) = (n_1,...,n_c)$. 
\end{proof}

For $(n_1,...,n_c) \in I_{n,c}$, denote $(S^{[n]})^\tau_{(n_1,...n_c)}$ the corresponding connected component of $(S^{[n]})^\tau$.

\begin{prop}

There is a canonical isomorphism of mixed Hodge structures
$$  H^{*} ( (S^{[n]})^\tau_{(n_1,...,n_c)} ) \cong  \bigoplus_{\substack{ \lambda \in \PP, \vec{\mu} \in \PP_{odd}^c \\ 2|\lambda|+|\vec{\mu}| =n\\  \ell(\vec{\mu}) = (n_1,...,n_c)}}    H^{ * + 2\ell(\lambda) +\ell(\vec{\mu})-n}(  B^{(\lambda)} \times C^{(\vec{\mu})}) \left(  \frac{2 \ell(\lambda)+\ell(\vec{\mu})-n}{2}\right) $$
\end{prop}

\begin{proof}
Following the proof of (\ref{canondecompofMHS}) we get an canonical isomorphism of mixed Hodge structures
$$  H^{*} ( (S^{[n]})^\tau_{(n_1,...,n_c)} ) \cong  \bigoplus_{(\lambda,\vec{\mu}) \in J}    H^{ * + 2\ell(\lambda) +\ell(\vec{\mu})-n}(  B^{(\lambda)} \times C^{(\vec{\mu})}) \left(  \frac{2 \ell(\lambda)+\ell(\vec{\mu})-n}{2}\right) $$
where $J$ the set of order pairs $ ( \lambda, \vec{\mu}) \in \PP \times \PP_{odd}^c$ with $2|\lambda|+|\vec{\mu}| =n$ for which the maximal dimension component of the fibre above   $(S^{(n)})^\tau_{\lambda, \vec{\mu}}$ is contained in $(S^{[n]})^\tau_{(n_1,...,n_c)}  $. 

Choose $ D \in (S^{(n)})^\tau_{\lambda, \vec{\mu}}$ and $Z \in \pi_\tau^{-1}(D)$.  By definition of  $(S^{(n)})^\tau_{\lambda, \vec{\mu}}$, the intersection of the support of $Z$ with $C_k$ is a set of $n_k$ points $\{q_1,...,q_{n_k}\}$. Furthermore, $Ind_Z(C_k) = \sum_{i=1}^{n_k} Ind_Z( q_i)$.  From Lemma \ref{maxcompind}, we see that $Ind_Z( q_i) = 1$ for each $q_i$, so $Ind_Z(C_k) = n_k$. 
\end{proof}

Recall from (\ref{stratoftauinv}) the stratification $ F^n  = \bigcup_{\lambda \in \PP_n} W_\lambda^{-,\tau} $. From Lemma \ref{maxcompfixedpoint2}, we know that $\overline{W_{(n)}^{-,\tau}}$ is only irreducible component of $F^n$ of top dimension.

\begin{lem}\label{maxcompind}
If $n$ is odd, and $Z \in (F^n)^\tau$ lies in the highest dimensional irreducible component, then $Ind_Z(Z) =1$. 
\end{lem}

\begin{proof}
Clearly that $\overline{W_{\lambda}^{-,\tau}}$ is connected for all $\lambda \in \PP_n$, and $Ind_Z(Z)$ is constant on connected components of $(F^n)^\tau$. At $Z_\lambda \in W_{(\lambda)}^{-,\tau} $ we have $ \mathcal{O}_{Z_\lambda} (Z_{\lambda}) \cong \C[x,y] / I_\lambda$, and under the action of $\tau$ this decomposes into  $+1$ and $-1$ eigenspaces, of multiplicities $(n + \ell_{odd}(\lambda))/2$ and $(n - \ell_{odd}(\lambda))/2$ respectively, where $\ell_{odd}(\lambda)$ is the number of parts $\lambda_j$ that are odd. Therefore  $Ind_{Z_\lambda}(Z_\lambda) = \ell_{odd}(\lambda) = Ind_{Z}(Z) $ for all $Z \in \overline{W_{\lambda}^{-,\tau}}$. In particular, when $n$ is odd  and $Z \in \overline{W_{(n)}^{-,\tau}}$ we have  $Ind_{Z}(Z) = 1$.  
\end{proof}

\section{Examples of anti-symplectic involutions on surfaces}\label{Examples of anti-symplectic involutions on surfaces}

Anti-symplectic involutions of K3 surfaces were classified by Nikulin \cite{Ni, Ni2} (nicely summarized in \cite{AST}). These are necessarily algebraic and are topologically determined by the isometry type of the $\tau$-invariant sublattice $NS^\tau$ of the Neron-Severi lattice $NS$ of the K3 surface. Such sublattices are classified by invariants $(r,a,\delta)$, where $r$ is the rank of $NS^\tau$, $2^a = |(NS^\tau)^*/NS^\tau|$ is its discriminant, and $\delta =0$ if $t^2 \in \Z$ for all $t \in (NS^\tau)^*$ and $\delta=1$ otherwise. There are 75 realizable invariants $(r,a,\delta)$, which can be found in (\cite{AST} Figure 1). In particular, the pair $(r,a)$ is realized if and only if
\begin{itemize}
\item $r \equiv a ~mod~2$, $a \leq r \leq 22-a$ and $r \geq 2$, or
\item $(r,a) \in \{ (1,1), (3,1), (9,1), (11,1), (17,1), (19,1),(2,0), (10,0), (18,0)\}.$ 
\end{itemize}

In these examples,

\begin{enumerate}
\item  If $(r,a,\delta) = (10,10,0)$, then $C = \emptyset$ and $B$ is an Enriques surface.
\item  If $(r,a,\delta)=(10,8,0)$, then $C$ is a disjoint union of two elliptic curves and $B$ is a rational surface with $\dim H^2(B) =10$.
\item In all remaining cases of $(r,a,\delta)$, $C$ is the disjoint union of a genus curve and $ k$-many rational curves, where $g= (22-r-a)/2$ and $k = (r-a)/2$, and $B$ is a rational surface with $\dim H^2(B) = r$.
 \end{enumerate}
The Hodge polynomial of $(S^{[n]})^\tau$ is easily determined from this classification using Theorem \ref{mainthm}.  

In \cite{FJM}, anti-symplectic involutions were considered in the context of mirror symmetry.  They prove that if $(S,\tau)$ and $(\check{S},\check{\tau})$ are a mirror pair of K3 surfaces with anti-symplectic involution, then their invariants $(r,a,\delta)$ are the same ( \cite{FJM} Example 5.1). An easy consequence is the following.

\begin{cor}
If $(S,\tau)$ and $(\check{S},\check{\tau})$ are a mirror pair of K3 surfaces with anti-symplectic involution, then for all $n\geq0$ there is an equality of Hodge polynomials
$$ MH( (S^{[n]})^\tau) = MH( (\check{S}^{[n]})^{\check{\tau}}). $$
\end{cor}





Another class of examples was the original motivation for this work. Let $\mathcal{C} \subseteq GL_n(\C)$ be a semisimple conjugacy class with determinant one and eigenvalue multiplicities $(n-1, 1)$.  Consider the the affine GIT quotient $\mathcal{M}_{n} =  U /\!\!/ PGL_n(\C)$, where $$U := \{  (A,B,C) \in GL_n(\C)^2 \times \mathcal{C} |  ABA^{-1}B^{-1} = C\},$$ and $PGL_n(\C)$ acts by conjugation. If $\Sigma$ is an genus one Riemann surface with a single point removed, then $\mathcal{M}_n$ parametrizes homomorphisms from $\pi_1(\Sigma)$ to $GL_n(\C)$ that send the positively oriented loop around the puncture to $\mathcal{C}$.  The character variety $\mathcal{M}_n$ is a smooth symplectic variety of dimension $2n$. 

When $n=1$, $\mathcal{M}_1 = \C^{\times} \times \C^{\times}$ with coordinate ring $\C[x^{\pm 1},y^{\pm 1}]$ and symplectic form $$\omega =  d log (x) \wedge d log (y)  = \frac{1}{xy} dx \wedge dy.$$ 
Hausel-Lettelier-Rodriguez-Villegas \cite{HLR} conjectured that  $\mathcal{M}_n$ has the same mixed Hodge polynomial as  $\mathcal{M}_1^{[n]}$  and that their E-polynomials agree, i.e.  $$MH(\mathcal{M}_1^{[n]}; x,y, -1) = MH(\mathcal{M}_n; x,y, -1).$$ More examples relating Hilbert schemes of points on surfaces with character varieties can be found in \cite{Gr}. 

An orientation reversing involution $\sigma$ of $\Sigma$ determines an anti-symplectic involution $\tau$ of $\mathcal{M}_n$ by precomposing with $\sigma_* \in Aut(\pi_1(\Sigma))$ and post-composing with the Cartan involution on $GL_n(\C)$ sending $A \mapsto (A^{-1})^T$  (see \cite{BW}). Examples of such involutions on $\mathcal{M}_1 = \C^{\times} \times \C^{\times} $ include $\tau(x,y) = (x^{-1},y)$ and $\tau(x,y) = (y^{-1},x^{-1})$. Inspired by \cite{HLR}, we had hoped that the mixed Hodge polynomials of $(\mathcal{M}_n)^\tau$ and $(\mathcal{M}_1^{[n]})^\tau$ would agree.  However, comparing our formula for $MH((\mathcal{M}_1^{[2]})^\tau )$ with a preliminary calculation of the E-polynomial $MH((\mathcal{M}_2)^\tau, x,y, -1)$ seems to show they do not equal.  We plan to return to this question in future work.


\begin{thebibliography}{999}

\bibitem{AST} Artebani, Michela, Alessandra Sarti, and Shingo Taki. ``K3 surfaces with non-symplectic automorphisms of prime order." Mathematische Zeitschrift 268 (2011): 507-533.

\bibitem{BW} Baird, Thomas, and Michael Wong. ``E-polynomials of character varieties for real curves." Transactions of the American Mathematical Society 376.12 (2023): 8657-8697.

\bibitem{B} Beauville, Arnaud. ``Vari\'et\'es K\"ahleriennes dont la premiere classe de Chern est nulle." Journal of Differential Geometry 18.4 (1983): 755-782.

\bibitem{BM} Borho, Walter and Robert MacPherson. ``Partial resolutions of nilpotent varieties." Ast\'erisque 101 (1983): 23-74.

\bibitem{ES} Ellingsrud, Geir, and Stein Arild Str\o mme. ``On the homology of the Hilbert scheme of points in the plane." Inventiones mathematicae 87 (1987): 343-352.


\bibitem{F} Fogarty, John. ``Algebraic families on an algebraic surface." American Journal of Mathematics 90.2 (1968): 511-521.

\bibitem{GS} G\"ottsche, Lothar, and Wolfgang Soergel. ``Perverse sheaves and the cohomology of Hilbert schemes of smooth algebraic surfaces." Mathematische Annalen 296 (1993): 235-245.

\bibitem{Gr} Groechenig, Michael. ``Hilbert schemes as moduli of Higgs bundles and local systems." International Mathematics Research Notices 2014.23 (2014): 6523-6575.

\bibitem{G} Grojnowski, Ian. ``Instantons and affine algebras I: the Hilbert scheme and vertex operators." Math. Res. Letters 275.3 (1996): 275-291.

\bibitem{HLR} Hausel, Tam\'as, Emmanuel Letellier, and Fernando Rodriguez-Villegas. ``Arithmetic harmonic analysis on character and quiver varieties II." Advances in Mathematics 234 (2013): 85-128.

\bibitem{H} Huybrechts, D. (2016). Automorphisms. In Lectures on K3 Surfaces (Cambridge Studies in Advanced Mathematics, pp. 330-357). Cambridge: Cambridge University Press. 


\bibitem{FJM} Franco, Emilio, Marcos Jardim, and Grégoire Menet. "Brane involutions on irreducible holomorphic symplectic manifolds." (2019): 195-235.



\bibitem{M} Macdonald, Ian G. ``Symmetric products of an algebraic curve." Topology 1.4 (1962): 319-343.

\bibitem{N} Nakajima, Hiraku. Lectures on Hilbert schemes of points on surfaces. No. 18. American Mathematical Soc., 1999.

\bibitem{N2} Nakajima, Hiraku. ``Heisenberg algebra and Hilbert schemes of points on projective surfaces." Annals of mathematics 145.2 (1997): 379-388.

\bibitem{Ni} Nikulin, Viacheslav V. ``Factor groups of groups of automorphisms of hyperbolic forms with respect to subgroups generated by 2-reflections. Algebrogeometric applications." Journal of Soviet Mathematics 22.4 (1983): 1401-1475.

\bibitem{Ni2} Nikulin, Viacheslav V. ``Discrete reflection groups in Lobachevsky spaces and algebraic surfaces." Proceedings of the International Congress of Mathematicians. Vol. 1. No. 2. 1986.

\bibitem{S} Serre, Jean-Pierre. Algebraic groups and class fields. Vol. 117. Springer Science \& Business Media, 2012.

\end{thebibliography}
\end{document}